\documentclass[final,3pt,twocolumn]{elsarticle}
\usepackage[UKenglish]{babel} 
\usepackage[T1]{fontenc}
\usepackage[utf8]{inputenc}
\usepackage{amsmath}
\usepackage{amsthm}
\usepackage{amssymb}
\usepackage{mathtools}
\usepackage{thmtools}
\usepackage{commath}
\usepackage{bbm} 
\usepackage{verbatim} 
\usepackage[dvipsnames]{xcolor} 
\usepackage{algorithmic}
\usepackage{algorithm}
\usepackage{cleveref}
\usepackage{comment}
\usepackage[caption=false,font=footnotesize]{subfig}
\usepackage{array}
\usepackage{algorithmic}
\usepackage{algorithm}
\usepackage{mathrsfs}
\usepackage[]{siunitx}
\usepackage{microtype}

\usepackage{tikz}
\usepackage{pgfplots}
\usepgfplotslibrary{external}
\usepgfplotslibrary{statistics}
\usetikzlibrary{decorations}
\usetikzlibrary{decorations.markings}
\usetikzlibrary{pgfplots.statistics}
\pgfplotsset{compat = 1.3}

\definecolor{myred}{rgb}{0.82, 0.1, 0.26} 
\definecolor{mygreen}{rgb}{0.0, 0.7, 0.0} 
\definecolor{myblue}{rgb}{0,0.4,1} 
\definecolor{myyellow}{rgb}{1,0.8,0.4} 
\definecolor{mybrown}{RGB}{169,94,43} 
\definecolor{mycyan}{RGB}{29,202,190} 
\definecolor{myviolet}{RGB}{223,133,223} 
\definecolor{mypink}{RGB}{255,51,153} 
\definecolor{mygreen2}{RGB}{153,230,93} 

\definecolor{AaltoYellow}{RGB}{255, 205, 0}
\definecolor{AaltoRed}{RGB}{239, 51, 64}
\definecolor{AaltoBlue}{RGB}{0, 94, 184}
\definecolor{AaltoPurple}{RGB}{117, 59, 189}
\definecolor{AaltoGreen}{RGB}{120, 190, 32}
\definecolor{AaltoPurple2}{RGB}{192, 26, 162}

\newcommand{\E}{\mathbb{E}}

\newcommand{\T}{\ensuremath{\mathsf{T}}}
\newcommand{\tr}{\ensuremath{\operatorname{tr}}}

\newtheorem{theorem}{Theorem}
\newtheorem{lemma}{Lemma}
\newtheorem{corollary}{Corollary}
\newtheorem{proposition}{Proposition}

\newtheorem{remark}{Remark}

\newif\ifdouble 
 \doubletrue

\begin{document}

\begin{frontmatter}

\title{Continuous-Discrete Filtering and Smoothing on Submanifolds of Euclidean Space}


\author{Filip~Tronarp and Simo~Särkkä}

\address{Department of Electrical Engineering and Automation, Aalto University, Rakentajanaukio 2, 02150 Espoo, Finland}

\begin{abstract}
In this paper the issue of filtering and smoothing in continuous discrete time is studied when the state variable evolves in some submanifold of Euclidean space, which may not have the usual Lebesgue measure. Formal expressions for prediction and smoothing problems are derived, which agree with the classical results except that the formal adjoint of the generator is different in general. For approximate filtering and smoothing the projection approach is taken, where it turns out that the prediction and smoothing equations are the same as in the case when the state variable evolves in Euclidean space. The approach is used to develop projection filters and smoothers based on the von Mises--Fisher distribution.
\end{abstract}


\end{frontmatter}

\section{Introduction}
\label{sec:intro}

Consider the continuous-discrete state space model:
\begin{subequations}\label{eq:pomp}
\begin{align}
\dif X(t) &= a(t,X(t)) \dif t + \sigma(t,X(t)) \dif W(t), \\
Y(t_n) &\mid X(t_n) \sim m(t_n,y \mid X(t_n)),
\end{align}
\end{subequations}
where $a \colon [0,T] \times \mathbb{R}^d \to \mathbb{R}^d$ is the drift,
$\sigma \colon [0,T] \times \mathbb{R}^d \to \mathbb{R}^{d \times q}$ is the diffusion coefficient, $W$ is a standard Wiener process on $\mathbb{R}^q$, and
$X$ is the state variable, which is measured by $\{Y(t_n)\}_{n=1}^N$ with measurement densities $m(t_n,y\mid X(t_n))$. The likelihood at time $t_n$ is denoted by $L(t_n,x)$ and the process noise covariance rate is denoted by $Q(t,x) = \sigma(t,x)\sigma^\T(t,x)$.

The problem of filtering and smoothing for the model in Eq. \eqref{eq:pomp} has been well studied when the filtering and smoothing distributions on $X(t)$ admit densities with respect to the Lebesgue measure on $\mathbb{R}^d$ \cite{Anderson1972,Leondes1970}. Implementation of the exact filtering and smoothing relations are in general intractable, with the notable exception of affine Gaussian systems \cite{Kalman1960,KalmanBucy1961,RauchTungStriebel1965}. Consequently approaches to approximate inference have been developed such as \emph{assumed density} \cite{Maybeck1982} and the \emph{projection approach} \cite{Brigo1999,Koyama2018}.

However, the continuous-discrete time inference problem is not as well studied for the case when the state $X$ is only supported on some submanifold $\mathbb{X}$ of $\mathbb{R}^d$, the assumed density approach has been taken for matrix Fisher distributions on the special orthogonal group $\mathbb{SO}(3)$ \cite{Lee2018b} and the von Mises--Fisher distribution on the unit sphere $\mathbb{S}^2$ in $\mathbb{R}^3$ \cite{tronarp2018b}. The discrete time problem has been given attention in, for example,  \cite{Bukal2017,Traa2014,Kurz2014a,Glover2014,Kurz2016a,Gilitschenski2016}.

The contribution of this paper is to establish the formal solutions to the filtering and smoothing problems when is supported on a submanifold $\mathbb{X}$ of $\mathbb{R}^d$ under the assumption that the filtering and smoothing distributions admit density with respect to some base measure $\lambda$ on $\mathbb{X}$. The solution formulae are similar to the classic Euclidean case with the difference being that the adjoint of the generator is taken in $\mathcal{L}_2(\mathbb{X},\lambda)$ rather than $\mathcal{L}_2(\mathbb{R}^d)$. Furthermore, the formal solutions are approximated using the projection method \cite{Brigo1999}. This gives the same approximation formulae as in the $\mathcal{L}_2(\mathbb{R}^d)$ case, thus generalising the result of \cite{Brigo1999,Koyama2018}.

The rest of this paper is organised as follows, the formal solutions to the filtering and smoothing problems are derived in Section \ref{sec:formal_solution}, in Section \ref{sec:projection} the formal solutions are approximated by the projection approach. The methodology is applied to reference vector tracking in Section \ref{sec:application} and conclusions are given in Section \ref{sec:conclusion}.

\section{Formal Solution}\label{sec:formal_solution}
In the following the set of measurements up to time $t$ is denoted by $\mathscr{Y}(t) = \{y(t_n) \colon t_n \leq t \}$, the filtering density is denoted by $p_F(t,x) = p(t,x\mid \mathscr{Y}(t))$, and the smoothing density is denoted by $p_S(t,x) = p(t,x\mid \mathscr{Y}(T))$. Recall that the generator of the It\^o  process $X$ is given by \cite{Sarkka2019}
\begin{equation}\label{eq:generator}
\mathcal{G}[\phi] = \sum_i  a_i \partial_i \phi + \frac{1}{2} \sum_{i,j} Q_{i,j}\partial_{i,j}^2 \phi,
\end{equation}
and its adjoint taken in $\mathcal{L}_2(\mathbb{X},\lambda)$ is denoted by $\mathcal{G}^a$. The present development does not require the explicit expression for $\mathcal{G}^a$ but it is used in the formal solution formulae. The Fokker--Planck equation on $\mathbb{X}$ with respect to $\lambda$ is given in Proposition \ref{prop:fokker}.
\begin{proposition}\label{prop:fokker}
The probability density for $X(t)$ evolves according to
\begin{equation}\label{eq:fokker}
\partial_t p = \mathcal{G}^a [p],
\end{equation}
where $\mathcal{G}^a$ is the adjoint of $\mathcal{G}$ taken in $\mathcal{L}_2(\mathbb{X},\lambda)$.
\end{proposition}

\begin{proof}
It\^o's formula implies that for arbitrary $\phi \in \mathcal{C}^2(\mathbb{X})$
\begin{equation*}
\begin{split}
\partial_t &\E[\phi(X(t))] =  \int_\mathbb{X} \phi(x) \partial_t p(t,x) \dif \lambda(x) = \E[\mathcal{G}[\phi](X(t))]\\
&= \int_\mathbb{X} \mathcal{G}[\phi](x) p(t,x) \dif \lambda(x) \\
&= \int_\mathbb{X} \phi(x)\mathcal{G}^a[p](t,x) \dif \lambda(x),
\end{split}
\end{equation*}
but $\phi$ is arbitrary and hence  $\partial_t p = \mathcal{G}^a [p]$.
\end{proof}
From the fact that $X$ is a Markov process and Proposition \ref{prop:fokker} it is clear that the filtering distribution, between measurements, evolves as
\begin{equation}\label{eq:formal_prediction}
\partial_t p_F = \mathcal{G}^a[p_F].
\end{equation}
The update is given by Baye's rule
\begin{equation}
p_F(t_n,x) = \frac{ L(t_n,x) p_F(t_n^-,x) }{ \int_\mathbb{X} L(t_n,x) p_F(t_n^-,x) \dif \lambda(x)}.
\end{equation}

The following Theorem was proved in \cite{Anderson1972} when the filtering and smoothing distributions of $X$ have densities with respect to the Lebesgue measure. 
\begin{theorem}\label{thm:formal_smoothing}
The smoothing density satisfies 
\begin{equation}\label{eq:formal_smoothing}
\partial_t p_S = \frac{p_S}{p_F} \mathcal{G}^a[p_F] - p_F \mathcal{G}\Big[ \frac{p_S}{p_F} \Big].
\end{equation}
\end{theorem}
While the proof of Theorem \ref{thm:formal_smoothing} is the same as in \cite{Anderson1972}, \emph{mutatis mutandis}, it is given in the following.
\begin{proof}
By the Markov property we have
\begin{equation*}
\begin{split}
&p_S(t,x) = \int_\mathbb{X} p_S(t,x \mid X(t+\dif t) = z) p_S(t+\dif t,z) \dif \lambda(z) \\
&=\int_\mathbb{X} p_F(t,x \mid X(t+\dif t) = z, \dif \mathscr{Y}(t)) p_S(t+\dif t,z) \dif \lambda(z)\\
&= \int_\mathbb{X} p_F(t,x \mid X(t+\dif t) = z, \dif \mathscr{Y}(t))\\
&\quad\times\big( p_S(t,z) + \partial_t p_S(t,z) \dif t \big)\dif \lambda(z) + o(\dif t).
\end{split}
\end{equation*}
The first term in this integral can be evaluated using Bayes' rule
\begin{equation}
\begin{split}
&p_F(t,x \mid X(t+\dif t) = z, \dif \mathscr{Y}(t)) \\
&= \frac{p(t+\dif t,z \mid X(t) = x, \dif \mathscr{Y}(t)) p_F(t,x\mid \dif\mathscr{Y}(t))}{p_F(t+\dif t,z)} \\
&= \frac{p(t+\dif t,z \mid X(t) = x ) p_F(t,x)}{p_F(t+\dif t,z)} \\
&=  \frac{\big( \delta(z-x) + \mathcal{G}^a[\delta(z-x)] \dif t \big) p_F(t,x)}{p_F(t+\dif t,z)} + o(\dif t),
\end{split}
\end{equation}
where the second equality follows from the fact that $\dif \mathscr{Y}(t) = \{y(t_n) \colon t \leq  t_n \leq t + \dif t \}$ is empty for sufficiently small $\dif t$. The last equality follows from Eq. \eqref{eq:fokker}.
Consequently
\begin{equation*}
\begin{split}
&p_S(t,x) = \frac{p_F(t,x)}{p_F(t+\dif t,x)} \big( p_S(t,x) + \partial_t p_S(t,x) \dif t  \big)\\
&+ p_F(t,x) \int_\mathbb{X} \mathcal{G}^a[\delta(z-x)] \frac{p_S(t,z)}{p_F(t+\dif t,z)} \dif t \dif \lambda(z)
+ o(\dif t)\\
&= p_S(t,x) + \partial_t p_S(t,x)\dif t - \frac{p_S(t,x)}{p_F(t,x)} \mathcal{G}^a[p_F](t,x) \dif t \\
&+ \mathcal{G}\Big[ \frac{p_S}{p_F}\Big](t,x)\dif t + o(\dif t),
\end{split}
\end{equation*}
and the conclusion follows.
\end{proof}

\section{The Projection Method}\label{sec:projection}
In the following the differential geometric setup of \cite{Brigo1999} is reviewed. Consider the metric space of square root densities $\mathcal{P}^{1/2}(\mathbb{X}) \subset \mathcal{L}_2(\mathbb{X},\lambda)$, for which the Hellinger metric is induced by the $\mathcal{L}_2(\mathbb{X},\lambda)$ norm. Furthermore, consider an $m$-dimensional manifold in $\mathcal{P}^{1/2}(\mathbb{X})$ with one global and smoothing coordinate chart
\begin{equation}
\mathcal{P}_\Theta^{1/2} = \{ \sqrt{p_\theta}, \quad \theta \in \Theta \subset \mathbb{R}^m \}.
\end{equation}
The tangent space at $\sqrt{p_\theta}$ is a closed subspace of $\mathcal{L}_2(\mathbb{X},\lambda)$, which is given by
\begin{equation}
\mathcal{T}_{\sqrt{p_\theta}} \mathcal{P}_\Theta^{1/2} = \operatorname{span}\{ \partial_1^\theta \sqrt{p_\theta}, \ldots \partial_m^\theta \sqrt{p_\theta} \}
\end{equation}
and the inner product  between basis elements of the tangent space is given by
\begin{equation}
\langle \partial_i^\theta \sqrt{p_\theta} , \partial_j^\theta \sqrt{p_\theta}\rangle = \frac{g(\theta)}{4},
\end{equation}
where $g(\theta)$ is the Fisher information matrix. The projection $\Pi_\theta \colon \mathcal{L}_2(\mathbb{X},\lambda) \mapsto \mathcal{T}_{\sqrt{p_\theta}} \mathcal{P}_\Theta^{1/2}$ is given by
\begin{equation}\label{eq:tp_projection}
\Pi_\theta v = 4 \sum_{i,j} g^{-1}_{i,j}(\theta) \langle v , \partial_j^\theta \sqrt{p_\theta}\rangle \partial_i^\theta \sqrt{p_\theta}.
\end{equation}
For particular forms of $v$ the projection formula simplifies according to the following Lemma \cite{Brigo1999}.
\begin{lemma}\label{lem:brigo}
Let $u \in \mathcal{L}_2(\mathbb{X},\lambda)$ satisfy $\E_\theta[\abs{u}] < \infty$. Then the projection of $v = \frac{1}{2}\sqrt{p_\theta} u$ onto $\mathcal{T}_{\sqrt{p_\theta}}$ is given by
\begin{equation}\label{eq:tp_projection2}
\Pi_\theta v =  \E_\theta[u \nabla_\theta \log p_\theta]^\T g^{-1}(\theta) \nabla_\theta \sqrt{p_\theta}.
\end{equation}
\end{lemma}

\subsection{Projection Filtering}
It follows from Eq. \eqref{eq:formal_prediction} that $\sqrt{p_F}$ is governed by
\begin{subequations}
\begin{align*}
\mathcal{F}^{1/2}[\phi] = \frac{\phi}{2\phi^2} \mathcal{G}^a[\phi^2], \\
\partial_t \sqrt{p_F} = \mathcal{F}^{1/2}[\sqrt{p_F}].
\end{align*}
\end{subequations}
Let $p_F(t_{n-1},x) \in \mathcal{P}_\Theta$, then the prediction formula from $t_{n-1}$ due to the projection approach is given by \cite{Brigo1999}
\begin{equation}\label{eq:def:projection_prediction}
\partial_t \sqrt{p_{\theta_F}} = \Pi_{\theta_F} \circ \mathcal{F}^{1/2}[\sqrt{p_{\theta_F}}].
\end{equation}

\begin{proposition}\label{prop:projection_prediction_parameter}
The curve in $\Theta$ defined by Eq. \eqref{eq:def:projection_prediction} is given by
\begin{equation}\label{eq:projection_prediction_parameter}
\dot{\theta}_F = g^{-1}(\theta_F) \E_{\theta_F}\big[\mathcal{G}[\nabla_{\theta_F}\log p_{\theta_F}]\big]
\end{equation}
\end{proposition}

Proposition \ref{prop:projection_prediction_parameter} was proved in \cite{Brigo1999} for the case when $p_F$ is a density with respect to the Lebesgue measure. The proof method is essentially the same but is given below for completeness. 

\begin{proof}
It follows from Lemma \ref{lem:brigo} that
\begin{equation*}
\begin{split}
\partial_t \sqrt{p_{\theta_F}} &= \Pi_{\theta_F} \circ \mathcal{F}^{1/2}[\sqrt{p_{\theta_F}}] \\
&= \E_{\theta_F}\Big[ \frac{\mathcal{G}^a[p_{\theta_F}]}{p_{\theta_F}} \nabla_{\theta_F} \log p_{\theta_F} \Big] g^{-1}(\theta_F) \nabla_{\theta_F}  \sqrt{p_{\theta_F}} \\
&= \E_{\theta_F}\big[\mathcal{G}[\nabla_{\theta_F}\log p_{\theta_F}]\big]^\T g^{-1}(\theta_F) \nabla_{\theta_F}   \sqrt{p_{\theta_F}}.
\end{split}
\end{equation*}
On the other hand, by the chain rule $\partial_t \sqrt{p}_{\theta_F} = \dot{\theta}_F^\T \nabla_{\theta_F} \sqrt{p_{\theta_F}}$, matching terms gives the result.

\end{proof}

\begin{corollary}\label{cor:projection_prediction_parameter}
Let $\mathcal{P}_\Theta$ be an exponential family, $p_\theta \propto \exp(\theta^\T s(x) - \psi(\theta))$. Then Eq. \eqref{eq:projection_prediction_parameter} reduces to
\begin{equation}
\dot{\theta}_F = g^{-1}(\theta_F) \E_{\theta_F}\big[\mathcal{G}[s]\big].
\end{equation}
\end{corollary}

\begin{proof}
It follows from Proposition \ref{prop:projection_prediction_parameter} and direct calculation.
\end{proof}

For the filter update, it is herein assumed that $\mathcal{P}_\Theta$ is a conjugate family for the likelihoods. For a projection approach to the filter update, see  \cite{Tronarp2019b}.

\subsection{Projection Smoothing}
It follows from Theorem \ref{thm:formal_smoothing} that $\sqrt{p_S}$ is governed by
\begin{subequations}
\begin{align*}
\mathcal{B}^{1/2}[\phi] &= \frac{\phi}{2\phi^2} \Big(  \frac{\phi^2}{p_F} \mathcal{G}^a[p_F] - p_F \mathcal{G}\Big[ \frac{\phi^2}{p_F} \Big]\Big), \\
\partial_t \sqrt{p_S} &= \mathcal{B}^{1/2}[\sqrt{p_S}].
\end{align*}
\end{subequations}
In order to arrive at a tractable algorithm, similarly to \cite{Koyama2018}, the operator $\mathcal{B}^{1/2}$ is approximated by
\begin{equation}
\widehat{\mathcal{B}}^{1/2}[\phi] = \frac{\phi}{2\phi^2} \Big(  \frac{\phi^2}{p_{\theta_F}} \mathcal{G}^a[p_{\theta_F}] - p_{\theta_F} \mathcal{G}\Big[ \frac{\phi^2}{p_{\theta_F}} \Big]\Big)
\end{equation}
and the projection smoother is given by
\begin{equation}\label{eq:def:projection_smoother}
\partial_t \sqrt{p_{\theta_S}} = \Pi_{\theta_S} \circ \widehat{\mathcal{B}}^{1/2}[\sqrt{p_{\theta_S}}].
\end{equation}
Proposition \ref{prop:projection_smoother_parameter} was proved in \cite{Koyama2018} for the case when $p_F$ and $p_S$ are densities with respect to the Lebesgue measure. The proof method is essentially the same but is given below for completeness. 
\begin{proposition}\label{prop:projection_smoother_parameter}
The curve in $\Theta$ defined by Eq. \eqref{eq:def:projection_smoother} is given by
\begin{equation}\label{eq:projection_smoother_parameter}
\begin{split}
\dot{\theta}_S &=    g^{-1}(\theta_S) \E_{\theta_S}\Big[ \nabla_{\theta_S} \frac{p_{\theta_F}}{p_{\theta_S}} \mathcal{G}\Big[\frac{p_{\theta_S}}{p_{\theta_F}} \Big]  \Big].
\end{split}
\end{equation}
\end{proposition}

\begin{proof}
It follows from Lemma \ref{lem:brigo} that
\begin{equation*}
\begin{split}
&\partial_t \sqrt{p_{\theta_S}} = \Pi_{\theta_S} \circ \widehat{\mathcal{B}}^{1/2}[\sqrt{p_{\theta_S}}] \\
&= \E_{\theta_S}\Bigg[ \frac{\mathcal{G}^a[p_{\theta_F}]}{p_{\theta_F}} \nabla_{\theta_S} \log p_{\theta_S} \Bigg]^\T g^{-1}(\theta_S)  \nabla_{\theta_S} \sqrt{p_{\theta_S}} \\
&- \E_{\theta_S}\Bigg[  \frac{p_{\theta_F}}{p_{\theta_S}}   \mathcal{G}\Big[ \frac{p_{\theta_S}}{p_{\theta_F}}\Big] \nabla_{\theta_S} \log p_{\theta_S} \Bigg]^\T g^{-1}(\theta_S)  \nabla_{\theta_S} \sqrt{p_{\theta_S}}.
\end{split}
\end{equation*}
The first expectation can be simplified according to
\begin{equation*}
\begin{split}
&\E_{\theta_S}\Bigg[ \frac{\mathcal{G}^a[p_{\theta_F}]}{p_{\theta_F}} \nabla_{\theta_S} \log p_{\theta_S} \Bigg] = \nabla_{\theta_S} \E_{\theta_S}\Bigg[ \frac{\mathcal{G}^a[p_{\theta_F}]}{p_{\theta_F}}\Bigg] \\
&=   \nabla_{\theta_S} \E_{\theta_F}\Bigg[ \mathcal{G}\Big[\frac{p_{\theta_S}}{p_{\theta_F}} \Big]  \Bigg] = \nabla_{\theta_S} \E_{\theta_S}\Bigg[ \frac{p_{\theta_F}}{p_{\theta_S}} \mathcal{G}\Big[\frac{p_{\theta_S}}{p_{\theta_F}} \Big]  \Bigg] \\
&= \E_{\theta_S}\Bigg[ \nabla_{\theta_S} \frac{p_{\theta_F}}{p_{\theta_S}} \mathcal{G}\Big[\frac{p_{\theta_S}}{p_{\theta_F}} \Big]  \Bigg]
\\
&\quad+ \E_{\theta_S}\Bigg[ \frac{p_{\theta_F}}{p_{\theta_S}} \mathcal{G}\Big[\frac{p_{\theta_S}}{p_{\theta_F}} \Big] \nabla_{\theta_S} \log p_{\theta_S}  \Bigg].
\end{split}
\end{equation*}
Inserting this into the previous equation gives
\begin{equation*}
\begin{split}
\partial_t \sqrt{p_{\theta_S}} &= \E_{\theta_S}\Big[ \nabla_{\theta_S} \frac{p_{\theta_F}}{p_{\theta_S}} \mathcal{G}\Big[\frac{p_{\theta_S}}{p_{\theta_F}} \Big]  \Big]^\T g^{-1}(\theta_S) \nabla_{\theta_S} \sqrt{p_{\theta_S}}
\end{split}
\end{equation*}
and the conclusion follows by the same argument as in the proof of Proposition \ref{prop:projection_prediction_parameter}.
\end{proof}

\begin{corollary}\label{cor:projection_smoother_parameter}
Let $\mathcal{P}_\Theta$ be an exponential family, $p_\theta \propto \exp(\theta^\T s(x) - \psi(\theta))$. Then Eq. \eqref{eq:projection_smoother_parameter} reduces to
\begin{equation}
\dot{\theta}_S = g^{-1}(\theta_S) \Big( \E_{\theta_S}\big[\mathcal{G}[s]\big] +  \E_{\theta_S}[J_s Q J_s^\T](\theta_S - \theta_F) \Big),
\end{equation}
where $J_s$ is the Jacobian of the sufficient statistic $s$.
\end{corollary}

\begin{proof}

It follows from Proposition \ref{prop:projection_smoother_parameter} and direct calculation.
\end{proof}

\section{Application: Tracking on $\mathbb{S}^2$}\label{sec:application}
Consider the following state space model
\begin{subequations}\label{eq:reference_vector_model}
\begin{align}
\dif &X = - \breve{\Omega}(t) \times X \dif t - \gamma^2 X \dif t+ \gamma X \times \dif W, \label{eq:reference_vector_dynamics}\\
&Y(t_n) \mid X(t_n) \sim \mathcal{N}(g X(t_n), \alpha^2 \mathrm{I}), \label{eq:reference_vector_measurement}
\end{align}
\end{subequations}
where  $\times$ denotes the vector cross-product and $X(t) \in \mathbb{S}^2$ for $t > 0$ if $X(0) \in \mathbb{S}^2$ \cite{tronarp2018b}. The model in Eq. \eqref{eq:reference_vector_model} can be used to track, for example, the local gravity direction by using gyroscope measurements $\breve{\Omega}$, accelerometer measurements $Y$, and setting $g$ to the local gravity constant. For this model $Q$ and $\mathcal{G}$ are given by
\begin{subequations}\label{eq:reference_vector_facts}
\begin{align}
Q(x) &= \gamma^2 ( \mathrm{I} \norm{x}^2 - xx^\T), \\
\mathcal{G}[\phi] &=   -(\breve{\Omega}(t) \times x + \gamma^2 x )^\T\nabla_x \phi + \frac{1}{2} \tr[Q \nabla_x^2 \phi ].
\end{align}
\end{subequations}
In order to develop projection filters and smoothers for this inference problem, a class of densities on the unit sphere needs to be selected. An obvious choice is the von Mises--Fisher family, which is an exponential family with respect to the uniform measure on $\mathbb{S}^2$ \cite{Mardia2000}. The von Mises--Fisher densities on $\mathbb{S}^2$ are given by
\begin{equation}
p_\theta(x) = \exp(\theta^\T x - \kappa(\norm{\theta})),
\end{equation}
where $\kappa(r) = -\log r + \log(4\pi) + \log \sinh r$ and the sufficient statistic is $s(x) = x$. Furthermore, for the von Mises--Fisher distribution the following holds
\cite{Mardia2000,Garcia2019b}
\begin{subequations}\label{eq:vmf_facts}
\begin{align}
\E_\theta[s(X)] &= \E_\theta[X]  = \kappa'(\norm{\theta}) \frac{\theta}{\norm{\theta}}, \\
g(\theta) &=  \frac{\kappa'(\norm{\theta})}{\norm{\theta}} P_\perp(\theta) + \kappa''(\norm{\theta})  P(\theta),  \\
g^{-1}(\theta) &= \frac{\norm{\theta}}{\kappa_3'(\norm{\theta})}P_\perp(\theta) + \frac{1}{\kappa''(\norm{\theta})}P(\theta), \label{eq:vmf:inverse_fisher}
\end{align}
\end{subequations}
where $P(\theta) = \theta \theta^\T /\norm{\theta}^2$ and $P_\perp(\theta) = \mathrm{I} - P(\theta)$.

\subsection{von Mises--Fisher Filtering and Smoothing}
Given the development in Section \ref{sec:projection}, deriving a von Mises--Fisher filter and smoother for the model in Eq. \eqref{eq:reference_vector_model} is straight-forward. The von Mises--Fisher family is conjugate to the likelihood (Proposition 1 in \cite{tronarp2018b}) and the prediction and smoothing formulae are given by Propositions \ref{prop:vmf_prediction} and \ref{prop:vmf_smoother}, respectively.

\begin{proposition}\label{prop:vmf_prediction}
The projection filter based on the von Mises--Fisher have the following prediction equation.
\begin{equation}\label{eq:vmf_prediction}
\dot{\theta}_F = - \breve{\Omega}(t) \times \theta_F - \frac{\gamma^2 \kappa'(\norm{\theta_F})}{\norm{\theta_F}\kappa''(\norm{\theta_F})} \theta_F
\end{equation}
\end{proposition}

\begin{proof}
It follows from Corollary \ref{cor:projection_prediction_parameter} by direct calculation using  Eq. \eqref{eq:reference_vector_facts} and Eq. \eqref{eq:vmf_facts}.
\end{proof}

\begin{remark}
By the change of parameters $\mu = \theta/\norm{\theta}$ and $\beta = \norm{\theta}$, it follows from Eq. \eqref{eq:vmf_prediction} that
\begin{subequations}
\begin{align}
\dot{\mu}_F &= - \breve{\Omega}(t) \times \mu_F, \\
\dot{\beta}_F &= - \frac{\gamma^2 \kappa'(\beta_F)}{\kappa''(\beta_F)},
\end{align}
\end{subequations}
which are the prediction equations for the filter in \cite{tronarp2018b}.
\end{remark}

\begin{proposition}\label{prop:vmf_smoother}
The von Mises--Fisher projection smoother formula is given by
\begin{subequations}
\begin{align}
G(\theta) &=  \frac{\gamma^2\norm{\theta}}{\kappa'(\norm{\theta})}P_\perp(\theta)
+  \frac{\gamma^2[1 - (\kappa'(\norm{\theta}))^2]}{\kappa''(\norm{\theta})}P(\theta) - \gamma^2 \mathrm{I},  \nonumber \\
\begin{split}
\dot{\theta}_S &= - \breve{\Omega}(t) \times \theta_S  - \frac{\gamma^2 \kappa'(\norm{\theta_S})}{\norm{\theta_S} \kappa''(\norm{\theta_S})} \theta_S\\
&\quad+ G(\theta_S) \big(\theta_S - \theta_F\big). \label{eq:vmf_smoother}
\end{split}
\end{align}
\end{subequations}
\end{proposition}
\begin{proof}
It follows from Corollary \ref{cor:projection_smoother_parameter} by direct calculation using  Eq. \eqref{eq:reference_vector_facts} and Eq. \eqref{eq:vmf_facts}.
\end{proof}

\subsection{Experimental Results}
The von Mises--Fisher filter and smoother (VMFF/VMFS) are evaluated on the model in Eq. \eqref{eq:reference_vector_model}, $g = 9.82$ corresponding to tracking the local gravity vector. The system is simulated one hundred times for all combinations of $s \in \{10^{-3},10^{-2}\}$ and $\alpha^2 \in \{10^{-3},10^{-2}\}$ with sample rate for $Y$ and $\breve{\Omega}$ at $100 \mathrm{Hz}$. The coordinates of $\breve{\Omega}$ are governed by zero mean Ornstein--Uhlenbeck processes with mean reversion rate $-5$ and diffusion constant $2.5$.

The von Mises--Fisher filter and smoother estimate the state by the mode $\theta/ \norm{\theta}$ and are compared against a Gaussian filter and  a Gaussian smoother (GF/GS) (Type II, see \cite{SarkkaSarmavuori2013}), which use norm-constrained minimum mean-square estimators. The von Mises--Fisher filter and smoother are initialised by the uniform distribution on $\mathbb{S}^2$ and the Gaussian estimators are initialised by moment matching the uniform distribution on $\mathbb{S}^2$.

The mean angular error is listed in Table \ref{tab:results} for the different parameters and methods. Clearly, there appears to be an advantage to appropriately accounting for the geometry of the state space, particularly as the system becomes noisier, where the Gaussian smoother can in fact produce worse estimates than the Gaussian filter.

\begin{table}\caption{Mean angular error (degrees).}\label{tab:results}
\centering
\begin{tabular}{|l|c|c|c|c|}\hline
$(\alpha^2,s)$ & VMFF & VMFS & GF & GS \\ \hline
$(10^{-3},10^{-3})$ & 1.3042 & \bf{0.9691} & 1.3083 & 1.1055 \\ \hline
$(10^{-2},10^{-3})$ & 2.3000 & \bf{1.6799} & 2.3094 & 1.7860 \\ \hline
$(10^{-3},10^{-2})$ & 3.5286 & \bf{2.9079} & 3.5619 & 4.3473 \\ \hline
$(10^{-2},10^{-2})$ & 6.8679 & \bf{5.0925} & 7.0990 & 7.6873 \\ \hline
\end{tabular}
\end{table}

\section{Conclusion}\label{sec:conclusion}
In this paper, the formal solutions to the filtering and smoothing problem have been established for the case when the state evolves in some submanifold of Euclidean space and the projection method was used to develop approximate solutions. Simulation studies suggest that there is indeed a benefit to appropriately accounting for the geometry of the state space, particularly for smoothing in noisy systems.

\section*{Acknowledgment}
Financial support by the Academy of Finland and Aalto ELEC Doctoral School is acknowledged. 

\section*{References}
\bibliographystyle{elsarticle-num-names}
\bibliography{../refs}

\end{document}